\documentclass[a4paper,10pt]{article}
\usepackage{amsmath,amsthm,amssymb}
\usepackage{comment}

\numberwithin{equation}{section}

\newcommand{\ot}{\otimes}
\newcommand{\id}{\mathord{\operatorname{id}}}
\newcommand{\C}{\mathbb{C}}
\newcommand{\irr}{{\rm Irr}}

\newcommand{\Mor}{{\rm Mor}}
\theoremstyle{plain}
\newtheorem{theorem}{Theorem}[section]
\newtheorem*{theoremA}{Theorem A}

\newtheorem{proposition}[theorem]{Proposition}

\theoremstyle{definition}

\begin{document}
%%%%%%%%%%%%%%%%%%%%%%%%%%%%%%%%%%%%%%%%%%%%%%%%%%%%%%%%%%%%%%%%%%%%%%%%%%%%%%%%%%%%%%%%%%

\begin{center}
{\LARGE\bf  Monoidal Rigidity for Free Wreath Products}

\bigskip

{\sc Pierre Fima and Lorenzo Pittau}%$^{(1,2)}$
\end{center}

\begin{abstract}
\noindent In this note we observe that any compact quantum group monoidally equivalent, in a nice way, to a free wreath product of a compact quantum group $G$ by the quantum automorphism group of a finite dimensional C*-algebra with a $\delta$-form is actually isomorphic to a free wreath product of $G$ by the quantum automorphism group of another finite dimensional C*-algebra with a $\delta$-form.
\end{abstract}

\section*{Introduction}

The theory of compact quantum groups, which includes the fundamental constructions of Drinfeld and Jimbo \cite{dri86,dri87,jim85} as $q$-deformation of the universal enveloping algebras of classical Lie algebras, was introduced by Woronowicz in the eighties \cite{wor87}. Later, Wang \cite{wan98} introduced the quantum automorphism group $G_{(B,\psi)}$ of a finite dimensional C*-algebra $B$ with a state $\psi$ and Banica \cite{ban99a} studied its representation category when $\psi$ is a $\delta$-form.

Bichon \cite{bic04} introduced the construction of free wreath product $G\wr_* S_n^+$ of a compact quantum group $G$ by the permutation quantum group $S_n^+$ by using an action of $S_n^+$ on $n$ copies of $G$. In analogy with the classical case, the free wreath product allows to describe the quantum symmetry group of $n$ copies of a finite graph in terms of the symmetry group of the graph and of $S_n^+$. The representation category of a free wreath product $G\wr_* S_n^+$, when $G$ is Kac, was computed in \cite{lt14} by Lemeux and Tarrago. The free wreath product $H^+_{(B,\psi)}(G)$ of a compact quantum and a discrete group has been introduced by the second author in the case of $G=\widehat{\Gamma}$, the dual of a discrete group $\Gamma$, and the representation category has been computed in \cite{pit14}. For a general $G$, the construction and the computation of the representation category has been studied in \cite{FP15}.

This note is a continuation of the general study of free wreath product initiated in \cite{FP15}. Since we know from \cite{FP15} the rigid C*-tensor category given by the representation theory of a free wreath product $H^+_{(B,\psi)}(G)$, we will focus now on compact quantum groups which are monoidally equivalent to such a free wreath product. The purpose of this note is to prove the following rigidity result.

\begin{theoremA}\label{ThmA}
Let $G,G_0$ be compact quantum groups and $(B,\psi)$ be a finite dimensional C*-algebra with a $\delta$-form $\psi$. If $G$ is monoidally equivalent to $H^+_{(B,\psi)}(G_0)$ via a {\rm nice enough} unitary tensor functor then there exists a finite dimensional C*-algebra $D$ with a $\delta$-form $\psi_D$ on $D$ such that $G$ is isomorphic to $H^+_{(D,\psi_D)}(G_0)$.
\end{theoremA} 

The precise meaning of \textit{nice enough} will be given in Section \ref{ProofA}, just before the statement of Theorem \ref{ThmIso}. When $G_0=\{e\}$ is the trivial (quantum) group, any unitary tensor functor is nice enough. However, when $G_0=\{e\}$, a stronger statement than the one of Theorem A is known to hold: it is proved in \cite{Mr15} that any compact quantum group having the same fusion rules as the one of $SO(3)$ is actually isomorphic to the quantum automorphism group of some finite dimensional C*-algebra.

The paper has two sections: Section \ref{Prel} is a preliminary section in which we introduce the notations of this paper and recall some useful results and Section \ref{ProofA} contains the proof of Theorem A (see Theorem \ref{ThmIso}).

\section{Preliminaries}\label{Prel}

In this paper, we always assume the scalar products on Hilbert spaces to be linear in the first variable. For Hilbert spaces $H,K$, we denote by $\mathcal{L}(H,K)$ the Banach space of bounded linear maps from $H$ to $K$ and we write $\mathcal{L}(H):=\mathcal{L}(H,H)$. The same symbol $\ot$ is used to denote the tensor product of Hilbert spaces or the minimal tensor product of C*-algebras.

Given a unital C*-algebra $A$ and unitaries $u\in\mathcal{L}(H)\ot A$ and $v\in\mathcal{L}(K)\ot A$, for $H,K$ Hilbert spaces, we define $\text{Mor}(u,v):=\{T\in\mathcal{L}(H,K)\,:\,(T\ot\ 1)u=v(T\ot 1)\}$. We also use the notation
$$u\ot v:=u_{13}v_{23}\in\mathcal{L}(H\ot K)\ot A,$$
where $v_{23}:=\id_H\ot v$ and $u_{13}:=(\sigma\ot\id_A)(\id_K\ot u)$ and $\sigma\,:\,\mathcal{L}(K)\ot\mathcal{L}(H)\rightarrow\mathcal{L}(H)\ot\mathcal{L}(K)$ is the flip automorphism.

For $G$ a compact quantum group in the sense of \cite{wor98} we denote by $C(G)$ the maximal C*-algebra of $G$ i.e. the enveloping C*-algebra of the $*$-algebra given by the linear span of the coefficients of the irreducible representations of $G$. We will also denote by $\varepsilon$ the trivial representation of $G$, $\irr(G)$ the set of equivalence classes of irreducible representations of $G$ and by $\mathcal{R}(G)$ the rigid C*-tensor category of finite dimensional unitary representations of $G$. We refer to the chapter $2$ of the book \cite{nt13} for all the details concerning rigid C*-tensor categories and monoidal equivalence.

Let $B$ be a finite dimensional C*-algebra and $\psi$ a faithful state on $B$. Then $B$ is a Hilbert space with the scalar product defined by $\langle a,b\rangle:=\psi(b^*a)$. Let $m_B\,:\, B\ot B\rightarrow B$ be the multiplication on $B$ and $\eta_B\,:\,\C\rightarrow B$ be the unit. We call $\psi$ a $\delta$\textit{-form} when $m_Bm_B^*=\delta\id_B$, where $m_B^*\,:\, B\rightarrow B\ot B$ is the adjoint of $m_B$ with respect to the scalar product on $B$ defined above.

Let $G_0$ be a compact quantum group and for $\alpha\in\irr(G_0)$, we choose a representative $u^\alpha$ acting on the Hilbert space $H_\alpha$. The free wreath product of $G_0$ by the quantum automorphism group $G_{(B,\psi)}$ of $(B,\psi)$, denoted by $H^+_{(B,\psi)}(G_0)$, is the compact quantum group defined in \cite{FP15} as follows. The full C*-algebra $C(H^+_{(B,\psi)}(G_0))$ of this quantum group is the universal unital C*-algebra generated by the coefficients of matrices $a(\alpha)\in\mathcal{L}(B\ot H_\alpha)\ot C(H^+_{(B,\psi)}(G_0))$, for $\alpha\in\irr(G_0)$, such that:
\begin{itemize}
\item $a(\alpha)$ is unitary for all $\alpha\in\irr(G_0)$,
\item $(m_B\ot S)\Sigma_{23}\in\text{Mor}(a(\alpha)\ot a(\beta),a(\gamma))$ for all $\alpha,\beta,\gamma\in\irr(G_0)$ and all $S\in\text{Mor}(\alpha\ot\beta,\gamma)$,
\item $\eta_B\in\text{Mor}(\varepsilon, a(\varepsilon))$.
\end{itemize}
Moreover, there exists a unique unital $*$-homomorphism $\Delta\,:\,C(H^+_{(B,\psi)}(G_0))\rightarrow C(H^+_{(B,\psi)}(G_0))\ot C(H^+_{(B,\psi)}(G_0))$ for which the elements $a(\alpha)$ are representations for all $\alpha\in\irr(G_0)$. The pair $(C(H^+_{(B,\psi)}(G_0)),\Delta)$ is a compact quantum group, called the free wreath product.

When $\psi$ is a $\delta$-form, the representation theory of $H^+_{(B,\psi)}(G_0)$ is totally understood \cite{FP15}. In particular, the dimension of the spaces $\text{Mor}(\varepsilon,a(\alpha_1)\ot\dots\ot a(\alpha_n))$ only depends on $G_0$ and $\alpha_1,\dots,\alpha_n$ and not on $(B,\psi)$.

\section{Proof of Theorem A}\label{ProofA}

Suppose that $G$ and $G_0$ are compact quantum groups and $G$ is unitary monoidally equivalent to $H^+_{(B,\psi)}(G_0)$, where $\psi$ is a $\delta$-form. This means (see \cite{nt13}) that there exist unitary tensor functors $F\,:\,\mathcal{R}(H^+_{(B,\psi)}(G_0))\rightarrow\mathcal{R}(G)$, $H\,:\,\mathcal{R}(G)\rightarrow\mathcal{R}(H^+_{(B,\psi)}(G_0))$ such that $FH$ and $HF$ are naturally monoidally unitarily isomorphic to the identity functors, meaning that the natural isomorphisms $FH\simeq\iota$ and $HF\simeq\iota$ are both implemented by unitaries.

For $\alpha\in\irr(G_0)$, we choose a representative $u^\alpha$ acting on the Hilbert space $H_\alpha$. Recall that the representation $a(\alpha)$ is acting on $B\ot H_\alpha$. We define the unitary representation $v_\alpha:=F(a(\alpha))$ of $G$ acting on the finite dimensional Hilbert space $K_\alpha$. For all $u,v$ finite dimensional unitary representations of $H^+_{(B,\psi)}(G_0)$, we denote by the same symbol $F_2\in{\rm Mor}(F(u)\ot F(v),F(u\ot v))$ the unitary given by the definition of the unitary tensor functor $F$. Since the category is strict we have:
\begin{equation}\label{Eq1}F_2\circ(F_2\ot \id)=F_2\circ(\id\ot F_2)\in\Mor(F(u)\ot F(v)\ot F(w),F(u\ot v\ot w))\text{ for all }u,v,w.\end{equation}
Moreover, the naturally of $F_2$ means that, for all $u,v,u',v'$ finite dimensional unitary representations we have:
\begin{equation}\label{Eq2}
 F(f\ot g)\circ F_2=F_2\circ(F(f)\ot F(g))\in\Mor(F(u)\ot F(v),F(u'\ot v'))\quad\text{for all }f\in\Mor(u,u'),g\in\Mor(v,v').
\end{equation}
Finally, we note that, since $F$ is unitary, we may and will assume that
\begin{equation}\label{Eq3}
F(\varepsilon)=\varepsilon\quad\text{and}\quad F(\id)=\id\text{ for all }\id\in\Mor(u,u).\end{equation}

Define the finite dimensional Hilbert space $D=K_\varepsilon$. Since $m_B\in\text{Mor}(a(\varepsilon)\ot a(\varepsilon),a(\varepsilon))$ and $\eta_B\in\text{Mor}(\varepsilon,a(\varepsilon))$, we may define the linear maps
$$m_D\,:\, D\ot D\rightarrow D,\,\, m_D:=F(m_B)F_2\quad\text{and}\quad\eta_D\,:\,\C\rightarrow D,\,\,\eta_D:= F(\eta_B).$$

\begin{proposition}\label{PropAlgebra} $(D,m_D,\eta_D)$ is a unital algebra and $m_Dm_D^*=\delta\id_D$.
\end{proposition}

\begin{proof}
Using Equations (\ref{Eq1}), (\ref{Eq2}), (\ref{Eq3}) and the associativity of $m_B$ we get:
\begin{eqnarray*}
m_D\circ(m_D\ot\id)&=&F(m_B)F_2\circ(F(m_B)F_2\ot\id)=F(m_B)\circ F_2\circ (F(m_B)\ot \id)\circ(F_2\ot 1)\\
&=&F(m_B)\circ F(m_B\ot \id)\circ F_2\circ(F_2\ot 1)=F(m_B\circ(m_B\ot\id))\circ F_2\circ(\id\ot F_2)\\
&=&F(m_B\circ(\id\ot m_B))\circ F_2\circ(\id\ot F_2)=F(m_B)\circ F(\id\ot m_B)\circ F_2 \circ(\id\ot F_2)\\
&=&F(m_B)\circ F_2\circ (\id\ot F(m_B))\circ(\id\ot F_2) =F(m_B)F_2\circ(\id\ot F(m_B)F_2)\\
&=&m_D\circ(\id\ot m_D).
\end{eqnarray*}
Hence, $m_D$ is associative. Moreover, using Equations $(\ref{Eq2})$, $(\ref{Eq3})$ and the the fact that $\eta_B$ is the unit of $(B,m_B)$, which means that $m_B(\eta_B\ot\id_B)=\id_B=m_B(\id_B\ot\eta_B)$, we find:
$$m_D(\eta_D\ot\id_D)=F(m_B)F_2(F(\eta_B)\ot\id_D)=F(m_B)F_2F_2^*F(\eta_B\ot\id_B)=F(m_B(\eta_B\ot\id_B))=F(\id_B)=\id_D\text{ and,}$$
$$m_D(\id_D\ot\eta_D)=F(m_B)F_2(\id_D\ot F(\eta_B))=F(m_B)F_2F_2^*F(\id_B\ot \eta_B)=F(m_B(\id_B\ot\eta_B))=F(\id_B)=\id_D.$$
Finally, $m_Dm_D^*=F(m_B)\circ F_2\circ F_2^*\circ F(m_B)^*=F(m_Bm_B^*)=\delta F(\id_B)=\delta\id_{D}$.
\end{proof}

We are now ready to turn $D$ into a $*$-algebra. Define $t=m_D^*\circ\eta_D=F_2^*F(m_B^*\eta_B)\in \mathcal{L}(\C,D\ot D)$. We may and will view $t\in D\ot D$. Define the anti-linear map $S_D\,:\,D\rightarrow D$ by $S_D(x)=(x^*\ot\id)(t)$. Denote by $L_x\in\mathcal{L}(D)$ the bounded operator given by left multiplication by $x\in D$ and write $1_D:=\eta_D(1)\in D$.

\begin{proposition}$S_D$ is an involution on the unital algebra $D$. Moreover, equipped with this involution and the norm defined by $\Vert x\Vert_D:=\Vert L_x\Vert$, for $x\in D$, $D$ is a unital C*-algebra and the following holds.
\begin{enumerate}
\item $(L_x)^*=L_{S_D(x)}$ for all $x\in D$.
\item $\psi_D\,:\,D\rightarrow\C$, $x\mapsto \langle x,1_D\rangle$ is a faithful state on $D$ satisfying $\psi_D(S_D(y)x)=\langle x,y\rangle$ for all $x,y\in D$.
\item $\psi_D$ is a $\delta$-form on $D$.
\end{enumerate}
\end{proposition}

\begin{proof}
$(1)$ Using diagrams computations one can easily check that:
$$(\id\ot m_B)\circ (m_B^*\circ\eta_B\ot\id)=m_B^*.$$
Applying the functor $F$ we find $(\id\ot m_D)\circ (m_D^*\circ\eta_D\ot\id)=m_D^*$. Hence, for all $x,y,z\in D$, one has
\begin{eqnarray*}
\langle L_{S_D(x)}(y),z\rangle&=&\langle m_D(S_D(x)\ot y),z\rangle= \langle m_D((x^*\ot\id)\circ m_D^*\circ\eta_D\ot y),z\rangle\\
&=&\langle(x^*\ot\id)(\id\ot m_D)\circ(m_D^*\circ\eta_D\ot\id)(y),z\rangle
=\langle (x^*\ot\id)(m_D^*(y)),z\rangle\\
&=&\langle m_D^*(y),x\ot z\rangle=\langle y,m_D(x\ot z)\rangle=\langle y,L_x(z)\rangle=\langle(L_x)^*(y),z\rangle
\end{eqnarray*}
This concludes the proof of $1$ and shows that $S_D$ is anti-multiplicative since it implies that, for all $x,y\in D$,
$$S_D(xy)=(L_{xy})^*(1_D)=(L_y)^*(L_x)^*(1_D)=L_{S_D(y)}(S_D(x))=S_D(y)S_D(x).$$
Assertion $(1)$ also shows that $S_D$ is involutive since it implies that $L_{S_D^2(x)}=(L_{S_D(x)})^*=((L_x)^*)^*=L_x$ hence, $S_D^2(x)=x$. It follows that $S_D$ turns $(D,m_D,\eta_D)$ into an involutive unital algebra. It is moreover clear that $\Vert x\Vert_D:=\Vert L_x\Vert_{\mathcal{L}(D)}$, for $x\in D$, defines a norm on $D$ which satisfies the C*-condition since we have $\Vert S_D(x)x\Vert_D=\Vert L_{S_D(x)x}\Vert=\Vert L_{S_D(x)} L_x\Vert=\Vert L_x^* L_x\Vert=\Vert L_x\Vert^2=\Vert x\Vert_D$ for all $x\in D$.

$(2)$ One has $\psi_D(S_D(y)x)=\langle S_D(y)x,1_D\rangle=\langle L_{S_D(y)}(x),1_D\rangle=\langle(L_y)^*(x),1_D\rangle=\langle x,L_y(1_D)\rangle=\langle x,y\rangle$. It follows directly from this relation and the non degeneracy of the scalar product on $D$ that $\psi_D$ is faithful. It also follows that the scalar product induced by $\psi_D$ on $D$ is the same as the initial scalar product on $D$. Hence, taking the adjoint with the $\psi_D$-scalar product we have, from Proposition \ref{PropAlgebra}, $m_Dm_D^*=\delta\id_D$.
\end{proof}

We assume from now that the functor $F$ is \textit{nice enough} meaning that for any $\alpha\in\irr(G_0)$, there exists a unitary $V_\alpha\in\mathcal{L}(K_\alpha,K_\varepsilon\ot H_\alpha)$, with $V_\varepsilon=\id_D$, and such that the following diagram is commutative for all $\alpha,\beta,\gamma\in\irr(G_0)$ and all $S\in\text{Mor}(\alpha\ot\beta,\gamma)$:

$$\begin{array}{ccc}
(K_\varepsilon\ot H_\alpha)\ot( K_\varepsilon\ot H_\beta)&\overset{(m_D\ot S)\Sigma_{23}}{\longrightarrow} & K_\varepsilon\ot H_\gamma\\
\uparrow V_\alpha\ot V_\beta& &\uparrow V_{\gamma}\\
K_\alpha\ot K_\beta&\overset{F((m_B\ot S)\Sigma_{23})F_2}{\longrightarrow} & K_{\gamma}\\
\end{array}$$

Observe that any unitary tensor functor is nice enough if $G_0$ is the trivial quantum group. We are now ready to state our main result. Let us denote by $\widetilde{a}(\alpha)\in \mathcal{L}( D\ot H_\alpha )\ot C(H_{(D,\psi_D)}^+(G_0))$ the canonical generators. 

\begin{theorem}\label{ThmIso}
There exists a unique $*$-isomorphism $\pi\,:\,C(H_{(D,\psi_D)}^+(G_0))\rightarrow C(G)$ such that
$$(\id\ot\pi)(\widetilde{a}(\alpha))=(V_\alpha\ot 1)v_\alpha(V_\alpha^*\ot 1)\in\mathcal{L}( D\ot H_\alpha )\ot C(G)\quad\text{for all  }\alpha\in\irr(G_0).$$
Moreover, $\pi$ intertwines the comultiplications.
\end{theorem}

\begin{proof}
\textbf{Step 1.} \textit{There exists a surjective unital $*$-homomorphism $\pi\,:\,C(H_{(D,\psi_D)}^+(G_0))\rightarrow C(G)$ such that $(\id\ot\pi)(\widetilde{a}(\alpha))=(V_\alpha\ot 1)v_\alpha(V_\alpha^*\ot 1)\in\mathcal{L}( D\ot H_\alpha )\ot C(G)$ for all  $\alpha\in\irr(G_0)$.}

From the universal property of the C*-algebra $C(H_{(D,\psi_D)}^+(G_0))$ it suffices to check that the unitary representations $u_\alpha:=(V_\alpha\ot 1)v_\alpha(V_\alpha^*\ot 1)$ of $G$ on the Hilbert space $D\ot H_\alpha$ satisfy the following relations:
\begin{enumerate}
\item $(m_D\ot S)\Sigma_{23}\in\text{Mor}(u_\alpha\ot u_\beta,u_{\gamma})$ for all $\alpha,\beta,\gamma\in\irr(G_0)$ and all $S\in\text{Mor}(\alpha\ot\beta,\gamma)$,
\item $\eta_D\in\text{Mor}(\varepsilon,u_{\varepsilon}))$.
\end{enumerate}

$(1)$ Since $F$ is nice enough we have, for all $\alpha,\beta,\gamma\in\irr(G_0)$ and all $S\in\text{Mor}(\alpha\ot\beta,\gamma)$,
$$V_\gamma^*(m_D\ot S)\Sigma_{23}(V_\alpha\ot V_{\beta})=F((m_B\ot S)\Sigma_{23})F_2\in\text{Mor}(v_\alpha\ot v_\beta,v_\gamma).$$
It follows that $(((m_D\ot S)\Sigma_{23})\ot 1)(u_\alpha\ot u_\beta)$ is equal to:
\begin{eqnarray*}
&&(((m_D\ot S)\Sigma_{23})\ot 1)(V_\alpha\ot 1\ot 1)(v_\alpha)_{13}(V_\alpha^*\ot 1\ot 1)(1\ot V_\beta\ot 1)(v_\beta)_{23}(1\ot V_\beta^*\ot 1)\\
&&=((((m_D\ot S)\Sigma_{23})(V_\alpha\ot V_\beta))\ot 1)(v_\alpha\ot v_\beta)(V_\alpha^*\ot V_\beta^*\ot 1)\\
&&=(V_{\gamma}\ot 1)v_{\gamma}(V_{\gamma}^*((m_D\ot S)\Sigma_{23})(V_\alpha\ot V_\beta)(V_\alpha^*\ot V_\beta^*)\ot 1)=u_{\gamma}((m_D\ot S)\Sigma_{23}\ot 1).
\end{eqnarray*}

$(2)$. It is obvious since $\eta_B\in\text{Mor}(\varepsilon,a(\varepsilon))$ implies that $\eta_D=F(\eta_B)\in\text{Mor}(\varepsilon,v_\varepsilon)$ and $u_\varepsilon=v_\varepsilon$.

Note that $\pi$ automatically intertwines the comultiplications and $\pi$ is surjective since $F$ is essentially surjective.

\textbf{Step 2.} \textit{$\pi$ is an isomorphism.}

It suffices to check that $\pi$ intertwines the Haar measures $h$ on $C(H_{(D,\psi_D)}^+(G_0))$ and $h_G$ on $C(G)$. Since the linear span of the coefficients of representations of the form $\widetilde{a}(\alpha_1)\ot\dots\ot\widetilde{a}(\alpha_n)$, for $\alpha_1,\dots, \alpha_n\in\irr(G_0)$ and $n\geq 1$, is dense in $C(H_{(D,\psi_D)}^+(G_0))$, it suffices to check that, for all $n\geq 1$, $\alpha_1,\dots,\alpha_n\in\irr(G_0)$ one has
$$(\id\ot h_G)(u_{\alpha_1}\ot\dots\ot u_{\alpha_n})=(\id\ot h)(\widetilde{a}(\alpha_1)\ot\dots\ot\widetilde{a}(\alpha_n)),$$
which is equivalent to $\text{dim}(\text{Mor}(\varepsilon,u_{\alpha_1}\ot\dots\ot u_{\alpha_n}))=\text{dim}(\text{Mor}(\varepsilon,\widetilde{a}(\alpha_1)\ot\dots\ot\widetilde{a}(\alpha_n))$. Since $u_\alpha\simeq v_\alpha=F(a(\alpha))$ for all $\alpha\in\irr(G_0)$ and $F$ is a monoidal equivalence, the left hand side is equal to
$$\text{dim}(\text{Mor}(\varepsilon,v_{\alpha_1}\ot\dots\ot v_{\alpha_n}))=\text{dim}(\text{Mor}(\varepsilon,a(\alpha_1)\ot\dots\ot a(\alpha_n)).$$
Moreover since, $\psi_D$ is a $\delta$-form we also know from \cite[Theorem 3.5]{FP15} an explicit formula for the number $\text{dim}(\text{Mor}(\varepsilon,a(\alpha_1)\ot\dots\ot a(\alpha_n))$ which only depends on $G_0$ and not on $(B,\psi)$ or $(D,\psi_D)$ and we have
$$\text{dim}(\text{Mor}(\varepsilon,a(\alpha_1)\ot\dots\ot a(\alpha_n))=\text{dim}(\text{Mor}(\varepsilon,\widetilde{a}(\alpha_1)\ot\dots\ot \widetilde{a}(\alpha_n)).$$
\end{proof}

\bibliographystyle{alpha}
\bibliography{biblio}

\noindent
{\sc Pierre FIMA} \\ \nopagebreak
  {Univ Paris Diderot, Sorbonne Paris Cit\'e, IMJ-PRG, UMR 7586, F-75013, Paris, France \\
  Sorbonne Universit\'es, UPMC Paris 06, UMR 7586, IMJ-PRG, F-75005, Paris, France \\
  CNRS, UMR 7586, IMJ-PRG, F-75005, Paris, France \\
\em E-mail address: \tt pierre.fima@imj-prg.fr}

\vspace{0.2cm}

\noindent
{\sc Lorenzo PITTAU} \\ \nopagebreak
 \em E-mail address: \tt lorenzopittau@gmail.com

\end{document}